\documentclass[a4paper,10pt]{article}
\usepackage[english]{babel}   
\usepackage[latin1]{inputenc}
\usepackage[dvips]{graphics}
\usepackage{subfigure}
\usepackage{graphicx}
\usepackage{epsfig}
\usepackage{hyperref}
\usepackage{framed}
\usepackage{psfrag}
\usepackage{tikz} 
\usepackage{amssymb}
\usepackage{amsmath}
\usepackage{lineno}
\usepackage{amsthm}
\usepackage[labelfont=bf]{caption}
\usepackage{xcolor}

\newtheorem{theorem}{Theorem}
\newtheorem{proposition}[theorem]{Proposition}
\newtheorem{lemma}[theorem]{Lemma}
\newtheorem{corollary}[theorem]{Corollary}
\newtheorem{ex}{Example}

\begin{document}

\title{On equilibria stability in an epidemiological SIR model with recovery-dependent infection rate}

\author{
Andres David Báez-Sánchez \thanks{Universidade Tecnológica Federal do Paraná (UTFPR), Câmpus Curitiba, Departamento Acadêmico de Matemática. Av. Sete de Setembro, 3165, 80230-901, Curitiba, PR, Brasil.} \thanks{adsanchez@utfpr.edu.br}
\\
Nara Bobko \footnotemark[1] \thanks{narabobko@utfpr.edu.br} 
}

\maketitle

\begin{abstract}

{\bf Abstract}.
We consider an epidemiological SIR model with an infection rate depending on the recovered population. 
We establish sufficient conditions for existence, uniqueness, and stability (local and global) of endemic equilibria and consider also the stability of the disease-free equilibrium. 
We show that, in contrast with classical SIR models, a system with a recovery-dependent infection rate can have multiple endemic stable equilibria (multistability) and multiple stable and unstable saddle points of equilibria. 
We establish conditions for the occurrence of these phenomena and illustrate the results with some examples.

\end{abstract}

\section{Introduction}
\label{intro}

Compartmental models, and particularly SIR models, have been extensively used for mathematical modeling of infectious diseases within a population~\cite{martcheva2015introduction,Brauer}. 

The main idea behind SIR models is to consider that a population $N$ is divided into three disjoint categories or compartments: susceptible individuals, infected individuals, and recovered or deceased individuals, denoted by $S$, $I$, and $R$, respectively, so that $N=S+I+R$.
Depending on the modeling approach, the variables $S$, $I$, and $R$ are considered to be the absolute numbers of individuals in each group or the proportion of individuals relative to the total population. In this work, we consider this latter approach.

Within these considerations, an epidemiological SIR model with vital dynamics and constant population can be stated as

\begin{equation}
\begin{split}
\dfrac{dS}{dt} &= \mu-\beta\,S\,I-\mu\,S\\
\dfrac{dI}{dt} &= \beta\,S\,I-\mu\,I-\gamma I\\
\dfrac{dR}{dt} &= \gamma \, I-\mu\,R\\
\end{split}
\label{model} 
\end{equation}
with $S(0) + I(0) + R(0) = 1$. 
The positive real numbers $\mu$, $\beta$, and $\gamma$ can be interpreted as birth-mortality rate, infection rate, and recovery rate, respectively. 
For more details about SIR models see for example~\cite{martcheva2015introduction}. {\color{black} Note that from~\eqref{model}, we can obtain 
$$\dfrac{dN}{dt} = \dfrac{dS}{dt} + \dfrac{dI}{dt} + \dfrac{dR}{dt} = \mu(1-N),$$
and since $N(t) \equiv 1$ is the only solution of this equation, satisfying $N(0) = 1$, we can consider  $N(t)=S(t) + I(t) + R(t) = 1$ for all $t$. }

Letting 
$$\tau=t\mu; \quad \widetilde{\beta}=\frac{\beta}{\mu};\quad \widetilde{\gamma}=\frac{\gamma}{\mu};\quad k=1+\widetilde{\gamma};\quad \text{ and }\quad R_0=\frac{\beta}{\mu+\gamma}=\frac{\widetilde{\beta}}{1+\widetilde{\gamma}}=\frac{\widetilde{\beta}}{k},$$

we obtain a redimensionalized version of~\eqref{model}:
\begin{equation*}
 \begin{split}
 \dfrac{dS}{d\tau} &= 1-kR_0\,S\,I-\,S\\
 \dfrac{dI}{d\tau} &= kR_0\,S\,I-kI\\
 \dfrac{dR}{d\tau} &= (k-1)\, I-R,\\
  \end{split}
\end{equation*}
with $S(0) + I(0) + R(0) = 1$.


Note that the parameters $\widetilde{\beta}, \widetilde{\gamma}, k$ and $R_0$ are all positive real numbers and in particular,  $k>1$. 
The parameter $R_0$ is called the basic reproduction number and its fundamental role in the description of the equilibria stability in the classical SIR model is well known~\cite{martcheva2015introduction}. 
$R_0$ can be interpreted as the number of cases one case generates, on average, in an uninfected population. It represents a measure of the effectiveness of the infection.

Several generalizations and modifications of the SIR model have been proposed by other authors, particularly considering non-constant epidemiological rates (see,~\cite{greenhalgh1995modeling,o1997epidemic,thieme2002endemic,liu2012infectious,alexander2005bifurcation,villavicencio2017backward,greenhalgh2017time,roberts2015nine}). 
These kinds of considerations have been recognized as necessary features to model more realistic epidemic situations, like the interaction between human behavior and disease dynamics~\cite{roberts2015nine,Manfredi}.

{\color{black} 
Consider, for example, the population behavior with respect to some possible anti-infection measures (like vaccination, quarantine or sexual precautions). 
The propagation of the disease can be affected by changes in the population behavior and, in the same way, the risk perception and state of the disease can influence the behavior of the population related to anti-infection measures \cite{ferrer2015risk}. 
Recent measles outbreaks, for example, are considered to be a direct consequence of the increasing number of unvaccinated children, due to parental behavior and beliefs~\cite{WHO}. 
In the case of antiretroviral therapy (ART) for HIV, patients under successful ART have lower morbidity and mortality rates and, in many cases, their viral load becomes so low that the patient can be considered  almost recovered. 
Massive scale-up of this successful ART has been considered one of the possible causes of an increase in the practice of sexual risky behaviors and as consequence, an increase in the number of HIV cases and other sexually transmitted diseases~\cite{boily2005impact,levy2011men,zaidi2013dramatic}. 
Most recently, in the context of the COVID-19 pandemic, human behavior has played a fundamental role in the diseases dynamics~\cite{Kraemereabb4218, Chinazzieaba9757, betsch2020how} and, at the same time, it has become evident that the increase in the number of infected  and death cases changed the way population and policymakers embrace anti-infection strategies~\cite{Remuzzi2020,Pisano,parkerNYT}.

In all the situations described above, infection rates changed during the evolution of the disease and, in the last two cases, these changes can be considered to be related to changes in the epidemiological variables $S, I$, or $R$. 
In the present paper, we are interested in the stability of equilibria in situations where the infection rate changes depending on the recovered/removed population.}
 
Hence, we propose the following generalized SIR model with a recovery-dependent infection rate:

\begin{equation}
 \begin{split}
 \dfrac{dS}{d\tau} &= 1-f(R)\,S\,I-\,S\\
 \dfrac{dI}{d\tau} &=f(R)\,S\,I-kI\\
 \dfrac{dR}{d\tau} &= (k-1)\, I-R,
  \end{split}
 \label{model4}
\end{equation}
{\color{black} where $f$ is a positive function of $R$, generalizing the infection rate $\beta$,} and $S(0)+I(0)+R(0)~=~1$. 
{\color{black} The function $f$  can be interpreted  as a quantification of the effect on the infection rate, produced by control strategies that depend on the size of $R$,  or, since $S+I+R=1$, that depend on  the  susceptible and infected population simultaneously.

It is worth noting that~\cite{o1997epidemic} is considered a deterministic model similar to~\eqref{model4}, and, although the main focus was on the stochastic version, a recurrent solution to the model was obtained when the recovery-dependent infection rate is considered piecewise constant. 

Our work focuses on the stability and multistability features of the equilibrium solutions of model~\eqref{model4}.}

The article is organized as follows: In Section~\ref{sec.simplified_model} we develop a {\color{black}two-dimensional} simplified model equivalent to~\eqref{model4} and, {\color{black}under additional conditions on $f$,} we prove several interesting results, including the non-existence of non-constant positive periodic solutions. 
In Section~\ref{sec.Disease_free}, the disease-free equilibrium is considered and two results about its local and global stability are established. 
The results of this section generalize well-known results for the classical SIR model. 

Section~\ref{sec.Endemic} consider{\color{black}s} endemic equilibrium points.  First, we define an auxiliary function $g$ and establish sufficient conditions for the existence of endemic equilibrium points in terms of $f$ and $g$.  
Later, we consider the local stability of endemic equilibrium points and we illustrate conditions for the occurrence of multiple locally stable endemic equilibrium points (multistability). 
Finally, we consider conditions for the uniqueness and global stability of an endemic equilibrium point.  
Final comments and concluding remarks are presented in Section~\ref{conclusion}.

\section{Simplified Model}
\label{sec.simplified_model}

In this section, we develop a simplified {\color{black}two-dimensional} model equivalent to~\eqref{model4}. 
In the following lemma, we show that model ~\eqref{model4} is well defined in the sense that, for all solutions, the conditions $S, I, R \in [0,1]$ and $S+I+R=1$ are preserved under the dynamics described in model ~\eqref{model4}.

\begin{lemma} \label{lemma_invariant} 
The set $\Omega = \{ S\geq 0, I\geq 0, R\geq 0 \text{ and } S+I+R=1\}$ is positively invariant under~\eqref{model4}.
\end{lemma}

\begin{proof} First, we consider the behavior of the solutions with some initial condition equal to $0$.

If $S(0)=0$, then $\dfrac{dS}{d\tau}(0) = 1 > 0$. 
If $I(0)=0$, then $\dfrac{dI}{d\tau}(0) = 0 $. 
If $R(0)=0$, then $\dfrac{dR}{d\tau}(0) = (k-1)I(0) \geq 0 $ since $k > 1$ and we consider $I(0) \geq 0$. 
This proves the positive invariance of the positive octant.

Consider now $N(t) = S(t) + I(t) +R(t)$. From~\eqref{model4} we have that

$$\dfrac{dN}{d\tau} = \dfrac{dS}{d\tau} + \dfrac{dI}{d\tau} +\dfrac{dR}{d\tau} = 1 - S - I - R = 1-N.$$

Since $N(0) = 1$, then the solution of the above {\color{black} ordinary differential equation} is $N(t) = 1$. 
That is,  $S(t) + I(t) +R(t) = 1$ for all $t \geq {\color{black} 0}$.
\end{proof}

Lemma~\ref{lemma_invariant} implies also that the solutions are bounded and, as a further consequence, we can consider $S = 1 - I - R$ to obtain the following simplified model:

\begin{equation}
 \begin{split}
 \dfrac{dI}{d\tau} &=I[f(R)\,(1-I-R)-k]\\
 \dfrac{dR}{d\tau} &= (k-1)\, I-R.\\
  \end{split}
 \label{model5}
\end{equation}

The study of the equilibrium points of~\eqref{model4} will be done through the study of the simplified model ~\eqref{model5}. 
Hence, it will be relevant to consider the associated Jacobian matrix given by: 

\begin{equation} J(I,R)=
\begin{bmatrix} 
    f(R)\,(1-I-R)-k-If(R)&\quad I\left[\displaystyle\frac{df}{dR}\cdot(1-I-R)-f(R)\right]\\
    (k-1)&-1
\end{bmatrix},
\label{jacobian}
\end{equation}
{\color{black} valid when $\frac{df}{dR}$ is well defined. 
In fact, assuming some additional conditions on $f$, we can obtain the following useful result.}

\begin{lemma} \label{noperiodic}   Let $f$ be a positive function, continuously differentiable on $\mathbb R$. The model given by~\eqref{model5} does not have non-constant periodic solutions with $0<I(t)$ for all $t$.
\end{lemma}

\begin{proof}
Consider $\phi$ given by $\phi(I,R)=\frac{1}{I}$ we have that 
\begin{equation}\label{Bendixson}
\dfrac{\partial}{\partial I}\left(\phi(I,R)\,I[f(R)\,(1-I-R)-k] \right)+\dfrac{\partial}{\partial R}\left(\phi(I,R)\,[(k-1)\, I-R]\right)=-f(R)-\frac{1}{I}<0,
\end{equation}
if $I>0$. From the Bendixson-Dulac criterion, it follows that the system does not have a non-constant periodic solution lying entirely in any simply connected region of the upper-plane $I>0$, so~\eqref{model5} does not have non-constant periodic solutions with $0<I(t)$ for all $t$.
\end{proof}

\section{Disease-free Equilibrium} 
\label{sec.Disease_free}

Note that $(I^*,R^*)=(0,0)$ is an equilibrium point of~\eqref{model5} corresponding to a  disease-free state. 
The following results  generalize well-known results related to the stability of the disease-free equilibrium in the classical SIR model~\cite{martcheva2015introduction}, considering $\frac{f(R)}{k}$ as a  \textit{variable} reproduction number $R_0$. 

\begin{lemma} \label{lemma_origem}  {\color{black} Let $f$ be a positive function, continuously differentiable on $\mathbb R$.} 
If $\frac{f(0)}{k}<1$, then $(0,0)$ is a locally stable equilibrium point of~\eqref{model5}. 
If $\frac{f(0)}{k}>1$, then $(0,0)$ is a local saddle equilibrium point.
\end{lemma}

\begin{proof} The results follow from~\eqref{jacobian}, since $J(0,0)=\begin{bmatrix}f(0)-k&0\\k-1&-1\\\end{bmatrix}$ has eigenvalues equal to $-1$ and $f(0)-k$.
\end{proof}

\begin{lemma}\label{lemma_origem_global} {\color{black} Let $f$ be a positive function, continuously differentiable on $\mathbb R$.} 
If  $\frac{f(0)}{k}<1$ and $(0,0)$ is the only equilibrium point of the model given by~\eqref{model5}, then $(0,0)$ is globally stable.
\end{lemma}

\begin{proof} Consider $Z=\{0\leq I\leq 1; 0\leq R \leq 1; I+R\leq 1\}$, and $X$ any open set on the plane such that $Z\subset X$. 
Because of Lemma~\ref{lemma_invariant}, any solution of~\eqref{model5} with initial conditions $u^0=(I(0),R(0))$ on $Z$, remains  bounded and the $\omega$-limit of $u_0$, $\omega(u^0)$, satisfies $\omega(u^0)\subset Z\subset X$. 
Because we are considering that $(0,0)$ is the only equilibrium point of~\eqref{model5}, from the Poincar\'e-Bendixson Theorem it follows that, for any initial condition $u^0=(I(0),R(0))\in Z$ we have that
\begin{enumerate}
\item $\omega(u^0)$ is a periodic orbit, or,
\item $(0,0)\in\omega(u^0)$.
\end{enumerate}

If $\omega(u^0)$ is a periodic orbit, Lemma~\ref{noperiodic} implies that~\eqref{model5} does not have periodic non-constant orbits with $I(t)>0$ for all $t$. 
Therefore,  if $\omega(u^0)$ is a periodic orbit, then the orbit must intercept the axis $I=0$. Equations~\eqref{model5} imply that, in this case, $I$ remains equal to zero and $R\to 0$ therefore $(I(t),R(t))\to (0,0)$. If $(0,0)\in\omega(u^0)$, then because $(0,0)$ is locally stable by Lemma~\ref{lemma_origem},  every solution that gets close enough, converges to $(0,0)$, so, in fact, in this case also $(I(t),R(t))\to (0,0)$.
\end{proof}

\section{Endemic Equilibrium}
\label{sec.Endemic}

\subsection{Characterization and Existence}
\label{subsec_characterization}
Now we consider the possibility of an endemic equilibrium point $(I^*,R^*)$, so $I^*>0$. 
Note that if $I^*>0$, then any endemic equilibrium point of~\eqref{model5} must satisfy the following equations:

\begin{equation}
f(R^*)\,(1-I^*-R^*)-k=0\quad \text{and}\quad (k-1)\, I^*-R^*=0;
\label{model5eq1}
\end{equation}
which can be rewritten in terms of $R^*$ as 

\begin{equation}
\label{model5eq2}
f(R^*)=\frac{k-1}{\frac{k-1}{k}-R^*}\quad \text{and}\quad I^*=\frac{1}{k-1}R^*.
\end{equation}

If we define the auxiliary function $g$  by

\begin{equation}\label{gfunction}
g(R)=\frac{k-1}{\frac{k-1}{k}-R},
\end{equation}
it is clear from~\eqref{model5eq2} that for the existence of endemic equilibrium, it is necessary that $f$ and $g$ intercept. 
In fact, it is possible to completely characterize the endemic equilibrium points of ~\eqref{model5} in terms of functions $f$ and $g$. 

\begin{theorem}
\label{theorema2}
Let $f$ be a positive function, differentiable on [0,1] and $g$  defined as in~\eqref{gfunction}. 
A point $(I^*,R^*)$ is an endemic equilibrium of~\eqref{model5} if and only if $R^*\in(0,\frac{k-1}{k})$, $I^*\in(0,\frac{1}{k})$, $I^*=\frac{1}{k-1}R^*$ and,  $f(R^*)=g(R^*)$. 
\end{theorem}

\begin{proof}
The results follow from the equivalence between Eqs.~\eqref{model5eq1} and~\eqref{model5eq2},  the fact that $g(R)$ is positive only if  $R<\frac{k-1}{k}$, and that $R^*\in(0,\frac{k-1}{k})$ if and only if $I^*=\frac{1}{k-1}R^*\in(0,\frac{1}{k})$  because $k>1$.  
\end{proof}

Theorem~\ref{theorema2} establishes that endemic equilibrium points occur if and only if the functions $f$ and $g$ intercept each other on $(0,\frac{k-1}{k})$. 
The next corollary establishes a simple condition to ensure that this interception will occur.

\begin{theorem} \label{existence} Let $f$ be a positive function on $\mathbb R$, differentiable $[0,1]$. 
If $f(R)>g(R)$ for some $R\in [0,\frac{k-1}{k})$ then~\eqref{model5} has at least one endemic equilibrium point $(I^*,R^*)$,  with $R^*\in(R,\frac{k-1}{k})$ and $I^*\in(\frac{R}{k-1},\frac{1}{k})$. 
In particular, if $f(0)>k$, there exists at least one endemic equilibrium.
\end{theorem}

\begin{proof}
Consider the function $h=f-g$. 
Note that, because $f$ and $g$ are continuous on $[0,\frac{k-1}{k})$, $h$ is also continuous on $[0,\frac{k-1}{k})$. 
According to Theorem~\ref{theorema2}, to obtain the desired result, we must prove that $h$ has at least one root on $(R,\frac{k-1}{k})$. 
If $f(R)>g(R)$ for some $R\in [0,\frac{k-1}{k})$ then $h(R)>0$, and by the hypothesis on $f$ and the definition of $g$ we have $\displaystyle\lim_{R\to\left(\frac{k-1}{k}\right)^-} h(R)=f\left(\frac{k-1}{k}\right)-\lim_{R\to\left(\frac{k-1}{k}\right)^-}g(R)=-\infty$. 
From the Mean Value Theorem and the continuity of $h$, the previous statements imply that $h$ has at least one root $R^*\in (R,\frac{k-1}{k})$. 
By making $I^*=\frac{R^*}{k-1}$ we obtain the desired equilibrium point as $(I^*,R^*)$. 
The final statement follows from the fact that $g(0)=k$.
\end{proof}

\subsection{Local Stability of Endemic Equilibrium}
\label{subsec_localstability}

Theorems~\ref{theorema2} and~\ref{existence} establish conditions to verify the existence of endemic equilibrium points in terms of functions $f$ and $g$. 
The next theorem shows that the relationship between the derivatives of $f$ and $g$ can be used to classify the local stability of the endemic equilibrium obtained. 
Note that for all $R\not=\frac{k-1}{k}$, $g$ satisfies  $\dfrac{dg}{dR}=\dfrac{1}{k-1}g^2(R)$.
 
\begin{theorem}\label{theorema1} Let $f$ be a positive function, differentiable on $[0,1]$; $g$ defined as in~\eqref{gfunction}; and $(I^*,R^*)$  an endemic equilibrium point of~\eqref{model5}. If 
\begin{equation}
\frac{df}{dR}(R^*)<\frac{dg}{dR}(R^*) \left(\text{ or } <\frac{1}{k-1}g^2(R^*) \text{ or } <\frac{1}{k-1}f^2(R^*)\right),
\label{diffineq}
\end{equation}
 then $(I^*,R^*)$ is a locally stable equilibrium point. If 
 \begin{equation}
\frac{df}{dR}(R^*)>\frac{dg}{dR}(R^*) \left(\text{ or } >\frac{1}{k-1}g^2(R^*) \text{ or } >\frac{1}{k-1}f^2(R^*)\right),
\label{diffineqgeq}
\end{equation}
then $(I^*,R^*)$ is a locally saddle point.
\end{theorem}

\begin{proof}
If $(I^*,R^*)$ is an endemic equilibrium point of~\eqref{model5}, then, from Theorem~\ref{theorema2}, we have that $f(R^*)=g(R^*)$. 
Because $\dfrac{dg}{dR}=\dfrac{1}{k-1}g^2(R)$, we have that $\frac{dg}{dR}(R^*)=\frac{1}{k-1}g^2(R^*)=\frac{1}{k-1}f^2(R^*)$. 
Therefore, to obtain the desired result, we can use any of these three equivalent expressions. 
We will use $\frac{1}{k-1}f^2(R^*)$.

Since $k>0$, the first equation in~\eqref{model5eq1} implies that $1-I^*-R^*\not=0$ and $f(R^*)\not=0$. 
Using~\eqref{model5eq1}, we have that the Jacobian matrix~\eqref{jacobian} evaluated on $(I^*,R^*)$ is given by:

 \begin{equation} 
J(I^*,R^*) = 
\begin{bmatrix} 
 \displaystyle -I^*f(R^*)&I^*\left[\frac{df}{dR}(R^*)\frac{k}{f(R^*)}-f(R^*)\right]\\
(k-1)&-1
\end{bmatrix}.
\label{jacobianeq}
\end{equation}

By the Routh-Hurwitz criterion, in order to prove the local stability of $(I^*,R^*)$, it would be sufficient to show that the characteristic polynomial of the Jacobian matrix~\eqref{jacobianeq} has positive coefficients. 
Therefore, it would be sufficient to prove that:

\begin{align}
 \label{aux1} \text{(Trace):  }& I^*f(R^*)+1>0.\\
 \label{aux2} \text{(Determinant):  }& I^*f(R^*)-(k-1)I^*\left[\displaystyle\frac{df}{dR}(R^*)\frac{k}{f(R^*)}-f(R^*)\right]>0.
\end{align}

Inequality~\eqref{aux1} is satisfied because we are considering $f$ as a positive function. 
As $I^*\not=0$, the inequality given by~\eqref{aux2} is equivalent to the following inequalities:

\begin{align*}
\displaystyle 0&<f(R^*)-(k-1)\left[\frac{df}{dR}(R^*)\frac{k}{f(R^*)}-f(R^*)\right],\\
\displaystyle 0&<kf(R^*)-(k-1)\frac{df}{dR}(R^*)\frac{k}{f(R^*)},\\
\displaystyle	(k-1)\frac{df}{dR}(R^*)\frac{k}{f(R^*)}&<kf(R^*),\\
\displaystyle	\frac{df}{dR}(R^*)&<\frac{1}{k-1}f^2(R^*).
\end{align*}
	
Hence, if $\dfrac{df}{dR}(R^*)<\dfrac{1}{k-1}f^2(R^*)$, the characteristic polynomial of matrix~\eqref{jacobianeq} has only positive coefficients, which implies that both eigenvalues have a negative real part. 
So, $(I^*,R^*)$ is a locally stable equilibrium point.

Similarly, if $\frac{df}{dR}(R^*)>\frac{1}{k-1}f^2(R^*)$, then matrix~\eqref{jacobianeq} has a negative determinant, and its characteristic polynomial has the form $\lambda^2+b\lambda-c$ with $b,c>0$. 
This implies that matrix~\eqref{jacobianeq}  has one positive and one negative real eigenvalue and, therefore, $(I^*,R^*)$ is a locally saddle point.
\end{proof}

\begin{corollary}\label{decreasing} 
Let $f$ be a positive function, differentiable on $[0,1]$ and let $(I^*,R^*)$ be an endemic equilibrium point of~\eqref{model5} such that $\frac{df}{dR}(R^*)\leq 0 $. 
Then $(I^*,R^*)$ is locally stable. 
\end{corollary}

\subsection{Multiple Endemic Equilibrium and local Multistability}
\label{Multistability}

Theorem~\ref{theorema2} implies that multiple endemic equilibrium points occur if $f$ and $g$ have multiple interception points on $(0,\frac{k}{k-1})$. 
The following result shows, that under some conditions, the existence of one endemic equilibrium implies the existence of another one.

\begin{proposition}\label{multy} 
Let $f$ be a positive function, differentiable on $[0,1]$; $g$ defined as in~\eqref{gfunction}; and $(I^*,R^*)$ and endemic equilibrium of ~\eqref{model5} such that $\frac{df}{dR}(R^*)\not=\frac{dg}{dR}(R^*)$. 
Then $(I^*,R^*)$ is locally stable or there exists another endemic equilibrium point $(\overline{I}^*,\overline{R}^*)$ with $\overline{R}^*>R^*$.
\end{proposition}

\begin{proof}
Because $(I^*,R^*)$ is an endemic equilibrium of~\eqref{model5}, we have $f(R^*)=g(R^*)$ and, because $\frac{df}{dR}(R^*)\not=\frac{dg}{dR}(R^*)$, then by Theorem~\ref{theorema1} either is locally stable if $\frac{df}{dR}(R^*)<\frac{dg}{dR}(R^*)$ or, locally unstable  if $\frac{df}{dR}(R^*)>\frac{dg}{dR}(R^*)$. 
In the last case, for values of $R$ slightly bigger than $R^*$, the function $f$ will be greater than $g$ so Theorem~\ref{existence} implies the existence of at least one endemic equilibrium point $(\overline{I}^*,\overline{R}^*)$ with $\overline{R}^*>R^*$.
\end{proof}

The following example illustrates a situation with multiple unstable and stable endemic equilibrium points.

\begin{ex}\label{example1}
Let $n$ be a fixed positive integer, $R^*_i~=~i\left(\dfrac{k-1}{k}\right)\dfrac{1}{2n}$  for $i=0, 1, \ldots, 2n-1$ and $g(R)=\frac{k-1}{\frac{k-1}{k}-R}$. 
Let $f$ be a positive and differentiable function on $[0,1]$ such that, $f(0)<k$ and $f(R) = g(R) - \sin\left(2n\pi\frac{k}{k-1}R\right)$ for $R \in [R^*_1,R^*_{2n-1}]$ (Figure~\ref{fig1}).
Note that 
$
f(R^*_i) = g(R^*_i) - \sin\left(i\pi\right) = g(R^*_i).
$
Therefore, Theorem~\ref{theorema2} implies that, in this case, model~\eqref{model5} has at least $2n$ equilibrium points given by $(R_i^*,\frac{1}{k-1}R_i^*)$.
Furthermore, 
$$
\frac{df}{dR}(R^*_i) = \frac{dg}{dR}(R^*_i) - \frac{2n\pi k}{k-1}\cos\left(i\pi\right).
$$
Hence, for $i = 2, 4, \ldots, 2n-2$, we have $\frac{df}{dR}(R^*_i) = \frac{dg}{dR}(R^*_i) - \frac{2n\pi k}{k-1} < \frac{dg}{dR}(R^*_i)$ and, from Theorem~\ref{theorema1},  these $n-1$ equilibrium points are locally stable. 
On the other hand, for $i = 1, 3, \ldots, 2n-1$, we have $\frac{df}{dR}(R^*_i) = \frac{dg}{dR}(R^*_i) + \frac{2n\pi k}{k-1} > \frac{dg}{dR}(R^*_i)$; which, by Theorem~\ref{theorema1}, implies that these $n$ equilibrium points are locally unstable saddle points. 
Since $f(0) < k$,  Lemma~\ref{lemma_origem} implies the local stability of equilibrium point $(0,0)$. 
This alternation between stable and unstable saddle points is illustrated in Figure~\ref{fig1}. 
\end{ex}

Note that in Example~\ref{example1}, the \textit{generalized} variable reproduction number $\frac{f(R)}{k}$ may attain some values less than $1$ and, in a similar fashion as in the backward bifurcation phenomenon~\cite{gumel2012causes,zhang2019backward}, the disease-free equilibrium  co-exists with several endemic locally stable equilibrium points.  
{\color{black} The multistability phenomenon, i.e., the coexistence of different stable equilibrium points for a given set of parameters, has been the focus of a lot of research in the applied dynamical systems community. 
In multistable systems, the asymptotic  behavior depends crucially on the initial conditions. 
For an overview of instances of multistability across different areas, and an extensive list of references see~\cite{PISARCHIK2014167}.}


\begin{figure}[ht!]
\centering
	\begin{minipage}{1\linewidth}
	\centering
		\subfigure[]{
			\centering
			\includegraphics[width=8.9cm]{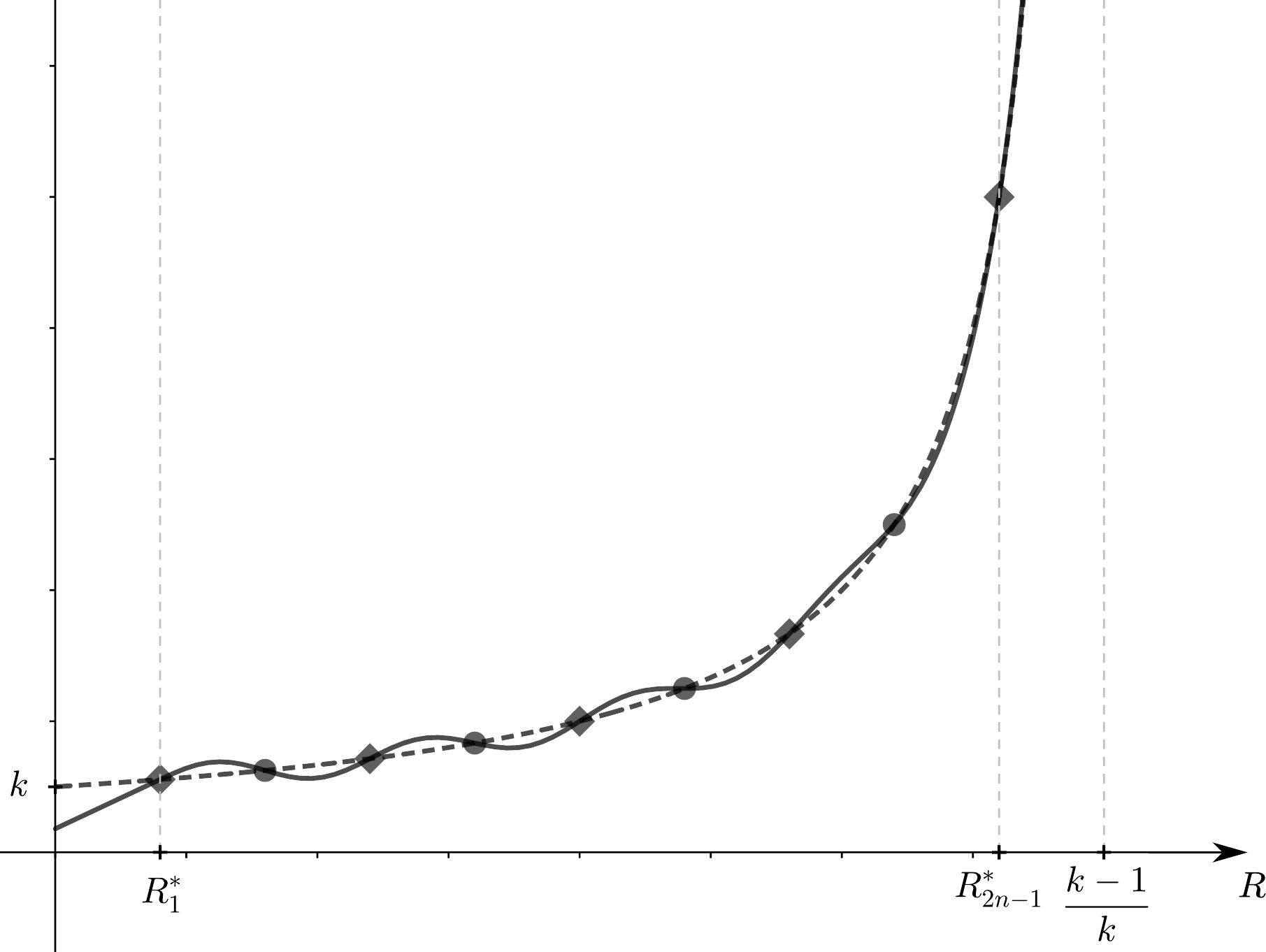}
		\label{fig1a}
		}
	\end{minipage}
		
	\vspace{0.5cm}
	
	\begin{minipage}{1\linewidth}
	\centering
		\subfigure[]{
			\centering
			\includegraphics[width=8.9cm]{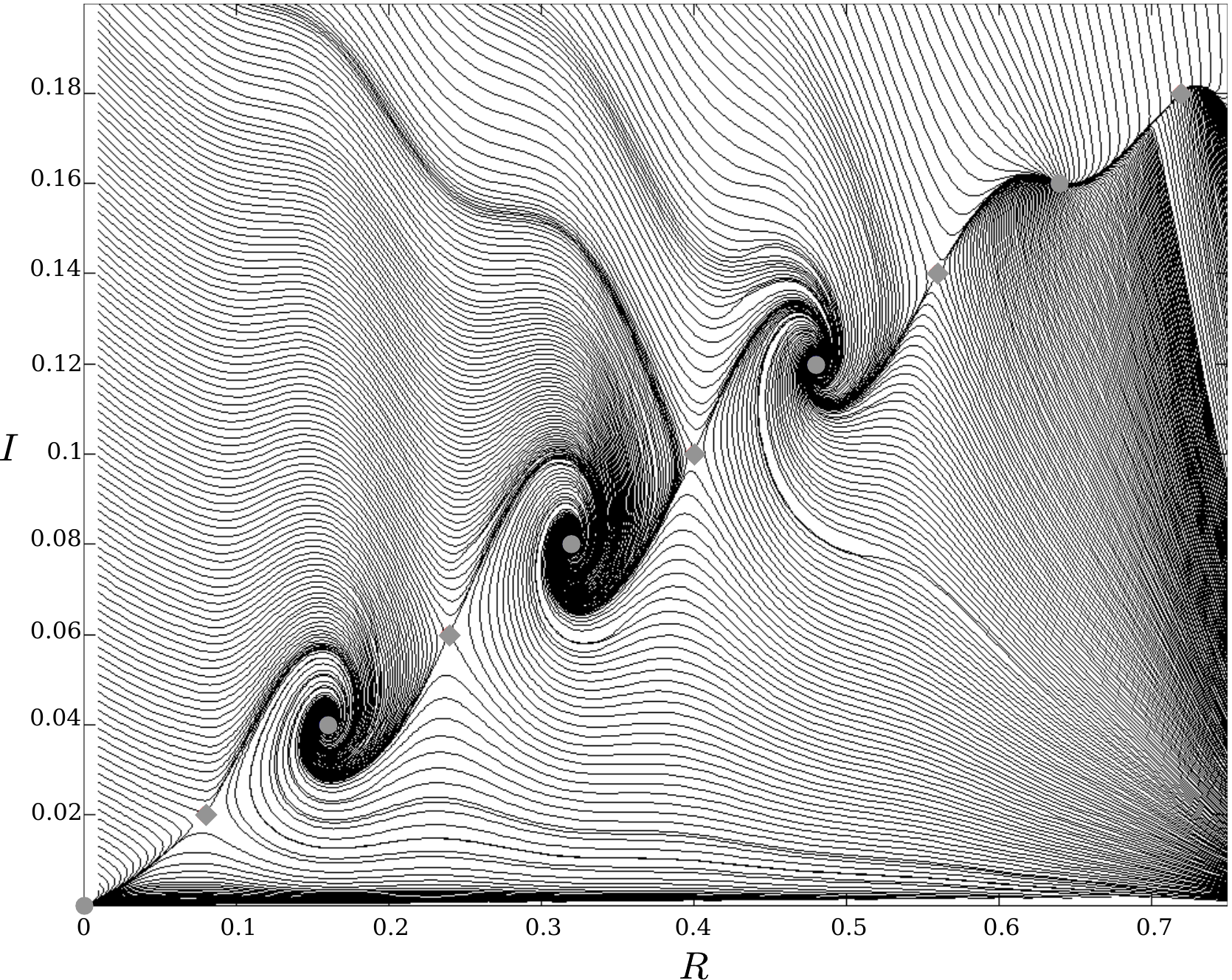}
		\label{fig1b}
		}
	\end{minipage}
	\caption{{\bf Multistability}\label{fig1}. 
	Consider $n=5$ and $k=5$ in Example~\ref{example1}. Diamond markers correspond to saddle equilibrium points and circle markers correspond to locally stable equilibrium points. 
	Functions $f$ (dashed line) and $g$ (solid line) are pictured on {\color{black} the Figure~\ref{fig1a}.
    On Figure~\ref{fig1b}, the partial phase plane $R\times I$ for~\eqref{model5} is pictured.}}
\end{figure}

\subsection{Uniqueness of Endemic Equilibrium and Global Stability}
\label{subsec_unique_global}

In this final subsection, we consider the possibility of global stability for endemic equilibrium. 
Clearly, this is only possible if there exists only one locally stable endemic equilibrium. 
The next result establishes a sufficient condition to guarantee such a situation.

\begin{proposition}\label{unique}
Let $f$ be a positive differentiable function on $[0,1]$ such that $f$ is constant or $f$ is non-increasing .
If $\frac{f(0)}{k}<1$ then~\eqref{model5} has no endemic equilibrium points. 
If $\frac{f(0)}{k}>1$ then there exists a unique endemic equilibrium point for~\eqref{model5}. Furthermore, the endemic equilibrium point is locally stable.
\end{proposition}

\begin{proof}

Consider as in Theorem~\ref{existence}, that $h=f-g$ with $g(R)=\frac{k-1}{\frac{k-1}{k}-R}$. 
Note that because $g$ is strictly increasing and $f$ is constant or non-increasing, the function $h$ is strictly decreasing. 
Note also that $g(0)=k$, so if $\frac{f(0)}{k}<1$ then $f(0)<g(0)$ and therefore $h(0)<0$. 
Because $h$ is strictly decreasing, $h$ has no roots. 
By Theorem~\ref{theorema2}, this means that there are no endemic equilibrium points for~\eqref{model5}.

If $\frac{f(0)}{k}>1$, then, by Theorem~\ref{existence}, the system~\eqref{model5} has at least one endemic equilibrium $(I^*,R^*)$ and therefore $h$ has at least one root $R^*\in (0,\frac{k-1}{k})$. Since $h$ is strictly decreasing, the root must be unique, so the endemic equilibrium is also unique. 
The local stability follows from Corollary~\ref{decreasing}. 
\end{proof}

The conditions for $f$ in Proposition~\ref{unique}, are not necessary for the uniqueness of a locally stable endemic equilibrium point, because it is possible to find a positive and strictly increasing function $f$ that intercepts $g$ only once, as illustrated in the following example.

\begin{ex}\label{example2}
Let $f:[0,1] \to \mathbb R$ be given by $f(R) = kR^2 + 2k$. 
Note that $f$ is a positive, differentiable, and strictly increasing function on $[0,1]$.
Furthermore, $f(R) = g(R)$ has a unique solution $R^* \in \left[0,\frac{k-1}{k}\right)$, which corresponds to a locally stable endemic equilibrium point (Figure~\ref{fig2}).


\begin{figure}[ht!]
\centering
	\begin{minipage}{\linewidth}
	\centering
		\subfigure[]{
			\centering
			\includegraphics[width=8.9cm]{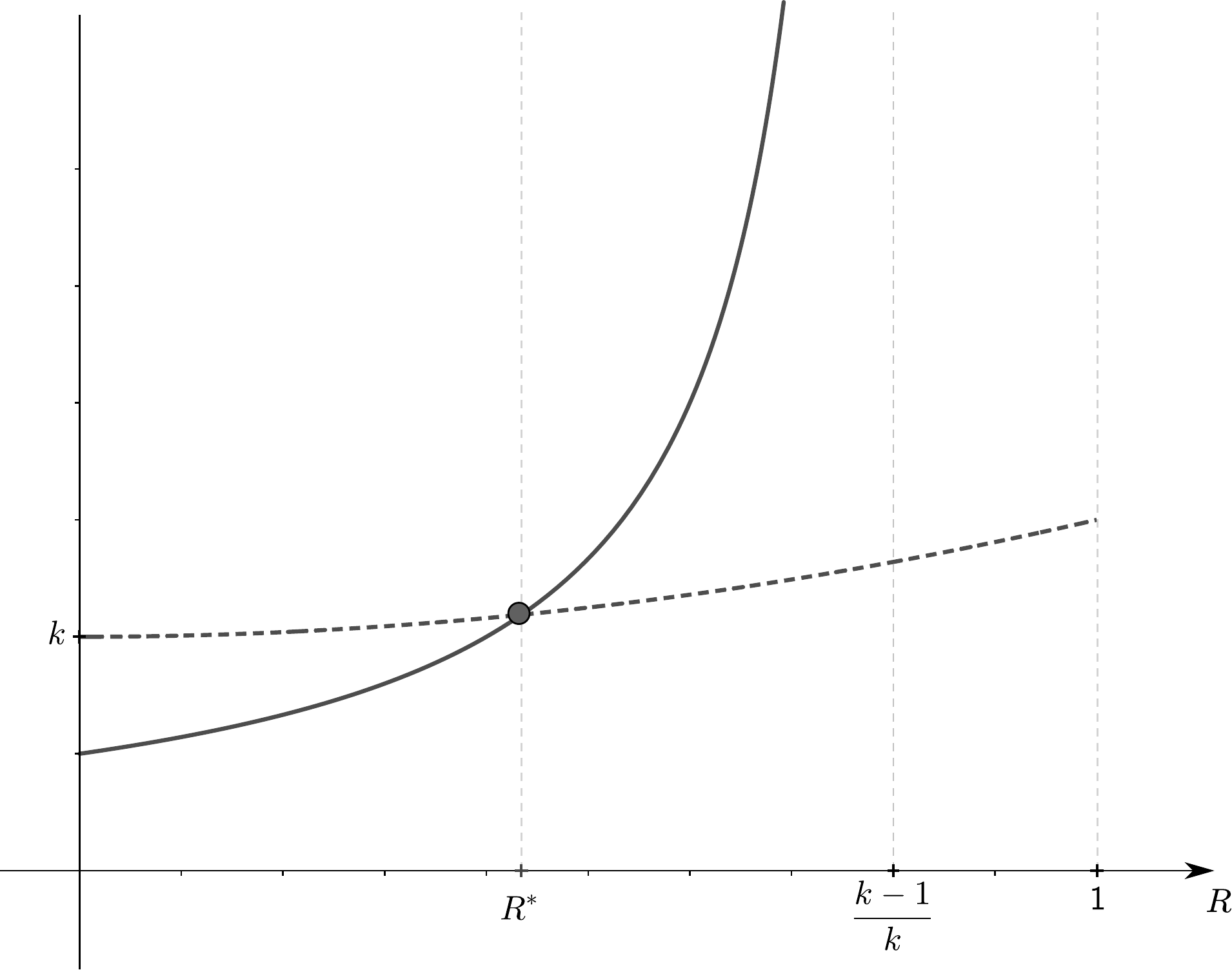}
		\label{fig2a}
		}
	\end{minipage}
	
	\vspace{0.5cm}
	
	\begin{minipage}{\linewidth}
	\centering
		\subfigure[]{
			\centering
			\includegraphics[width=8.9cm]{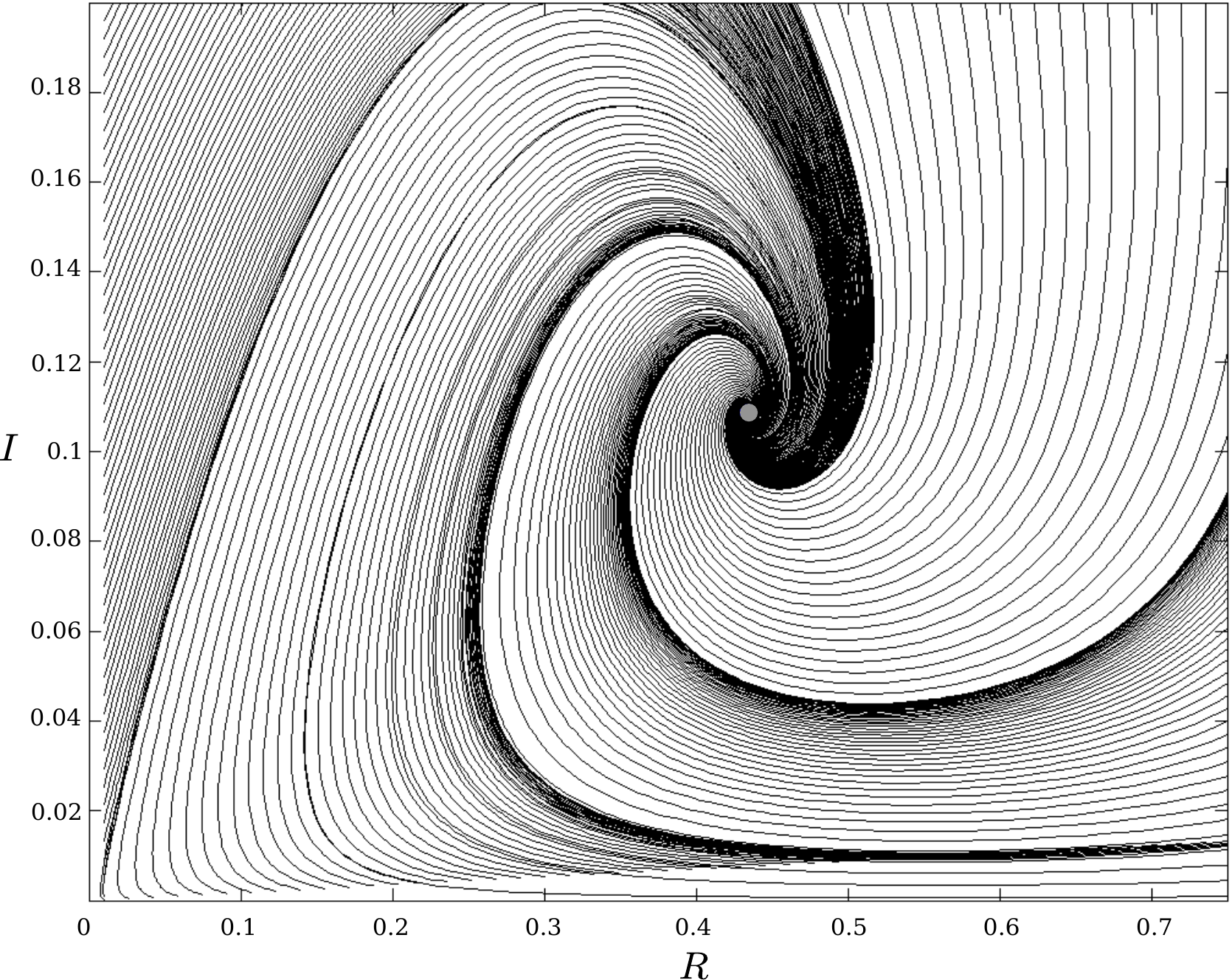}
		\label{fig2b}
		}
	\end{minipage}
	\caption{{\bf Unique and stable endemic equilibrium point}\label{fig2}. Consider $k=5$ in  Example~\ref{example2}. Functions $f$ (dashed line) and $g$ (solid line) are pictured on {\color{black} the Figure~\ref{fig2a}. 
    On Figure~\ref{fig2b}, the partial phase plane $R\times I$ for~\eqref{model5} is pictured. 
    The gray dot corresponds to the unique stable equilibrium point.}}
\end{figure}

\end{ex}

In the previous example, the partial phase-plane presented suggests that the endemic equilibrium is not only locally stable, but also  globally stable for initial conditions with $I(0)>0$. 
This, in fact, is true as a consequence of the following theorem.

\begin{theorem}\label{Global} Let $f$ be a positive function and  continuously differentiable on $\mathbb R$  with $f(0)>k$. If model ~\eqref{model5} has a unique endemic equilibrium point $(I^*,R^*)$  and $\frac{df(R^*)}{dR}\not=\frac{dg(R^*)}{dR}$, then $(I^*,R^*)$ is globally stable for $I(0)>0$.
\end{theorem}

\begin{proof}
Note that because we are considering a unique endemic equilibrium point $(I^*,R^*)$, $f$ and $g$ intercept only once at $R^*$. 
Because $f(0)>k=g(0)$, the continuity of $f$ and $g$ imply that $f(R)>g(R)$ if $R<R^*$ and, if $R>R^*$, then $f(R)<g(R)$, otherwise Theorem~\ref{existence} implies the existence of another interception point for some $R>R^*$. 
Hence, 
\begin{align*}
\frac{df}{dR}(R^*)\!=\!\lim_{h\to 0}\frac{f(R^*\!+\!h)\!-\!f(R^*)}{h}\!&=\!\lim_{h\to 0}\frac{f(R^*\!+\!h)\!-\!g(R^*)}{h}\\&\leq\!\lim_{h\to 0}\frac{g(R^*\!+\!h)\!-\!g(R^*)}{h}\!=\!\frac{dg}{dR}(R^*).
\end{align*}
{\color{black}
Since we are assuming that $\frac{df(R^*)}{dR}\not=\frac{dg(R^*)}{dR}$, the above inequality implies that $\frac{df(R^*)}{dR}<\frac{dg(R^*)}{dR}$. 
From Theorem~\ref{theorema1}, it follows that the unique endemic equilibrium point $(I^*,R^*)$ is locally stable.}

We analyze now the  global stability of $(I^*,R^*)$. 
Note first that, because $f(0)>k$, the Lemma~\ref{lemma_origem} implies that the disease-free equilibrium $(0,0)$ is a saddle unstable point. 
We claim that the stable manifold associated with this disease-free saddle  equilibrium point corresponds to the axis $I=0$. 
Note first that if $I(0)=0$, then Eqs.~\eqref{model5} implies that $I(t)=0$ for all $t>0$ and $R(t)=R(0)e^{-t}\to 0$ as $t\to \infty$. 
Note also that, from the continuity of $f$ and the fact that $f(0)>k$, if $I$ and $R$ take small enough positive values, then from the first equation of~\eqref{model5} we have $\frac{dI}{d\tau}>0$. 
Therefore, if $(I(t),R(t))\to (0,0)$, then $I$ can not take values always strictly positive, which means that  $I(t_0)=0$ for some $t_0$. 
However, $(0,R(t_0)e^{-(t-t_0)})$ is a solution passing  through $(I(t_0),R(t_0))=(0,R(t_0))$. 
Because the uniqueness of the solution of~\eqref{model5}, it follows that $I(t)=0$ for all $t$. 
Hence, the stable manifold associated with $(0,0)$ is the axis $I=0$.

Consider now $Z=\{0\leq I\leq 1; 0\leq R \leq 1; I+R\leq 1\}$, and $X$ any open set on the plane such that $Z\subset X$. 
Because Lemma~\ref{lemma_invariant}, any solution of~\eqref{model5} with initial conditions $u^0=(I(0),R(0))$ on $Z$ remains  bounded, and the $\omega-$limit of $u_0$, $\omega(u^0)$ satisfy $\omega(u^0)\subset Z\subset X$. 
From the Poincar\'e-Bendixson Theorem, one of the followings holds:
\begin{enumerate}
\item $\omega(u^0)$ is a periodic orbit, or,
\item $\omega(u^0)$ a connected set composed of a finite number of fixed points together with homoclinic and heteroclinic orbits connecting these, or, 
\item $\omega(u^0)$ consists of an equilibrium.
\end{enumerate}

We claim that if $I(0)>0$ then $\omega(u^0)={(I^*,R^*)}$. 
Assume first that $I(0)>0$ and $\omega(u^0)$ is a periodic orbit. 
By Lemma~\ref{noperiodic}, the orbit $\omega(u^0)$ must intercept the axis $I=0$, but because this axis is the stable manifold of $(0,0)$, this implies that $I$ remains equal to zero and $R\to 0$, so it is impossible for $I$ to periodically return to $I(0)>0$. 
Hence, $\omega(u^0)$ can not be a periodic orbit.

We argue now that if $I(0)>0$, then the endemic equilibrium $(I^*,R^*)$ belongs to $\omega(u^0)$. If $(I^*,R^*)\not \in \omega(u^0)$, then there are not heteroclinic orbits and, because  the stable manifold of $(0,0)$ is the axis $I=0$, there are not homoclinic orbits either. 
Therefore, the only possibility is that $\omega(u^0)=\{(0,0)\}$. 
From the Bolzano-Weierstrass theorem and the compacity of $Z$, we would have that $\lim_{t\to\infty}u(t)=\lim_{t\to \infty}(I(t),R(t))=(0,0)$, but this implies that $I(0)=0$ which contradicts our hypothesis.  
Hence, if $I(0)>0$, then $(I^*,R^*)\in \omega(u^0)$.

Since the endemic equilibrium $(I^*,R^*)\in \omega(u^0)$ and is locally stable, every solution that gets close enough to it, converges to it. 
This implies that the disease-free equilibrium $(0,0)$ can not belong to $\omega(u^0)$, so we conclude that  if $I(0)>0$ then $\omega(u^0)$ consists only of the endemic equilibrium $(I^*,R^*)$ and therefore all solutions with $I(0)>0$ converge to $(I^*,R^*)$.
\end{proof}

\section{Conclusion}

\label{conclusion}

We considered a general SIR epidemiological model~\eqref{model4} with an infection rate $f$ depending on the recovered population $R$. 
The main contribution of this paper is the determination of sufficient conditions for the existence, uniqueness (Theorems~\ref{theorema2}, \ref{existence} and Proposition~\ref{unique}), and stability of endemic equilibrium points (Theorems~\ref{theorema1} and~\ref{Global}). 
These results are obtained in terms of $f$ and the auxiliary function $g(R)=\frac{k-1}{\frac{k-1}{k}-R}$. 

The auxiliary function $g(R)$ can be considered a recovery-dependent threshold for the infection rate $f(R)$.
If this threshold is surpassed in some state $(I,R)$, some of the consequences may be undesirable from an epidemiological point of view.  
Theorem~\ref{theorema2} implies that when $f(R)>g(R)$, then there must exist an endemic equilibrium point and Proposition~\ref{multy} implies that, in many situations, this equilibrium will be locally stable or it will lead to another endemic equilibrium point. 
On the other hand, the results in this paper imply that if one is able to control $f$ to remain less than $g$ for all $R$, then there are not endemic equilibrium points and the disease-free equilibrium will be globally stable. 

The relationship between $f$ and $g$ generalizes the relationship between the constants $R_0$ and $1$ in the classical SIR model and shows that when considering variable parameters in epidemiological models, one could expect the appearance of variable thresholds relevant to the effective control of the diseases. 

\bibliographystyle{ieeetr}

\end{document}